\documentclass[titlepage,12pt]{article}
\usepackage[top=2.5cm,left=2.5cm,right=2.5cm,bottom=2.5cm]{geometry}
\usepackage{amsmath}
\usepackage{amsthm}
\usepackage{amssymb}
\usepackage{amscd}
\usepackage{amsfonts}
\usepackage{amsbsy}
\usepackage{graphicx}
\usepackage[dvips]{psfrag}
\usepackage{setspace}
\usepackage{indentfirst}
\usepackage{rotating,graphics,psfrag,epsfig}
\usepackage{float}
\usepackage{hyperref}
\usepackage{booktabs}
\usepackage{xcolor}
\usepackage{mathrsfs}
\usepackage{overpic}
\usepackage{doi}
\usepackage{tikz}
\usepackage{lineno}

\makeatletter
\hypersetup{
		urlbordercolor=0 0 1,
		colorlinks=true,       	
    linkcolor=blue,
    citecolor=blue,
    filecolor=blue,
		urlcolor=blue,
		bookmarksdepth=4
	}

\newtheorem{theo}{Theorem}

\newtheorem{lem}{Lemma}
\renewenvironment{proof}{\noindent {\bfseries Proof.}}{\hfill $\blacksquare$\vspace{0.4cm}}

\begin{document}

\onehalfspacing

\begin{center}
{\Large{\bf Traveling wave solutions of the Burgers-Huxley equations}}
\end{center}

\bigskip

\begin{center}
{\large Luis Fernando Mello$^\ast$ and Ronisio Moises Ribeiro}\\
{\em Instituto de Matem\'atica e Computa\c c\~ao, Universidade Federal de
Itajub\'a,\\ Avenida BPS 1303, Pinheirinho, CEP 37.500-903, Itajub\'a,
MG, Brazil.\\}
e-mail: lfmelo@unifei.edu.br, roniribeiro@unifei.edu.br
\end{center}

\bigskip

\begin{center}
{\bf Abstract}
\end{center}

We study the traveling wave solutions of the Burgers-Huxley equation from a geometric point of view via the qualitative theory of ordinary differential equations. By using the Poincar\'e compactification we study the global phase portraits of a family of polynomial ordinary differential equations in the plane related to the Burgers-Huxley equation. We obtain the traveling wave solutions and their asymptotic behaviors from the orbits that connect equilibrium points taking into account the restrictions of the studied equation.

\bigskip

\noindent {\small {\bf Key-words}: Traveling wave solution, Burgers-Huxley equation, compactification, global phase portrait}

\noindent {\small {\bf 2020 Mathematics Subject Classification:} 34A34, 34A26, 34D23.}

\noindent \makebox [40mm]{\hrulefill}

\noindent {\footnotesize $^\ast$Corresponding author}.


\section{Introduction}\label{sec:1}

We study the traveling wave solutions of the nonlinear parabolic partial differential equation
\begin{equation}\label{eq1}
w_t + w^k w_z = w_{zz} + w^m (1 - w^n),
\end{equation}
where $w = w (z, t)$, $(z, t) \in \mathbb R \times (0, T)$, $T > 0$, and the parameters $k$, $m$, $n$ are positive integers. Equation \eqref{eq1} can be found in \cite{Acho} and there it is called the Burgers-Huxley equation. There are other nonlinear parabolic partial differential equations that receive this same name or sometimes the name generalized Burgers-Huxley equation. See \cite{KM, MPHLN} and references therein. Equation \eqref{eq1} is used to model the dynamics of various physical phenomena associated with reaction mechanisms. For instance, it describes the dynamics of electric pulses in nerve fibers and in the walls of liquid crystals \cite{Wang}, interactions between diffusion transport and convection or advection effects \cite{AF}.

In order to study solutions of \eqref{eq1}, consider the following change of variables
\begin{equation}\label{eq:n1}
\phi(\xi) = w(z, t),
\end{equation}
where $\xi = z - c\,t$ and $c>0$ is the velocity of the wave. If $\phi$ is a solution of \eqref{eq1} then it must satisfy the following second-order ordinary differential equation
\begin{equation}\label{eq:n2}
-c \phi' + \phi^k \phi' = \phi'' + \phi^m (1 - \phi^n),
\end{equation}
where the dash means derivative with respect to the variable $\xi$.

Defining the new variable $\psi = \phi'$, we obtain from \eqref{eq:n2} the following systems of two first-order ordinary differential equations
\begin{equation}\label{eq:n3}
\phi' = \psi, \quad \psi' = -c \psi + \phi^k \psi + \phi^m \left( \phi^n - 1 \right).
\end{equation}

From the point of view of ordinary differential equations, systems \eqref{eq:n3} form a four-parameter family of polynomial differential equations in the plane. In this context, it is possible to study their global phase portraits and, in particular, to study their equilibrium points, both finite and infinite.

Assume the existence of two equilibrium points (finite or infinite), denoted by $A = (\phi_A, \psi_A)$ and $B = (\phi_B, \psi_B)$. Assume further the existence of an initial condition $(\phi_0, \psi_0)$ such that the solution $(\phi(\xi), \psi(\xi))$ of \eqref{eq:n3} by $(\phi_0, \psi_0)$, defined in its maximal interval $(\xi_a, \xi_b) \subset \mathbb R$, satisfies
\[
\lim_{\xi \to \xi_a} (\phi(\xi), \psi(\xi)) = (\phi_A, \psi_A), \quad \lim_{\xi \to \xi_b} (\phi(\xi), \psi(\xi)) = (\phi_B, \psi_B).
\]
In this scenario, the first coordinate of the previous solution, $\phi(\xi)$, respecting the restrictions that the solutions of \eqref{eq1} may have, is called the traveling wave solution of equation \eqref{eq1}. See equation \eqref{eq:n1}. Thus, there is a close relationship between the traveling wave solutions of \eqref{eq1} and the solutions of \eqref{eq:n3} that tend to equilibrium points. See \cite{IS1}. However, it may happen that a solution of \eqref{eq:n3} that tends to equilibrium points does not generate a traveling wave solution of \eqref{eq1} due to restrictions imposed on their solutions. The most common of these restrictions is the boundedness of the solutions.

For simplicity and so that the size of this article is not too large, we assume $m=1$ and $n\in\{1,2\}$. The cases $m>1$ and $n>2$ are more complicated and will be studied in a future article. Defining $(x,y) = (\phi, \psi)$ and $s = \xi$, systems \eqref{eq:n3} are written as
\begin{equation}\label{maineq}
\dot{x} = y, \quad \dot{y} = -cy+x^ky+x \left( x^n-1 \right),
\end{equation}
where the dot means derivative means derivative with respect to the variable $s$.

Under the assumptions considered here, systems \eqref{maineq} are a three-parameter family of polynomial ordinary differential equations in the plane: $(x, y) \in \mathbb R^2$, $n\in\{1,2\}$, $k \in \mathbb N$ and $c > 0$. In this sense, instead of studying only the traveling wave solutions of \eqref{eq1}, we will conduct a more comprehensive study: we will study the global phase portraits of \eqref{maineq} and, within this study, we will highlight the traveling wave solutions of \eqref{eq1}. We believe that in this way the geometric aspects of the traveling wave solutions will be better understood.

The next theorem provides a characterization of the global phase portraits of systems \eqref{maineq}. In the statement of the theorem the word ``distinct'' means ``not topologically equivalent''.

\begin{theo}\label{thm:01}
Systems \eqref{maineq} have seven distinct phase portraits described below:
\begin{description}
\item[I.] Phase portrait $(i)$ of Figure \ref{gfp} if $n=1$, $k\geq1$ is odd and $0<c<1$;

\item[II.] Phase portrait $(ii)$ of Figure \ref{gfp} if $n=1$, $k\geq1$ is odd and $c=1$;

\item[III.1.] Phase portrait $(iii-1)$ of Figure \ref{gfp} if $n=1$, $k\geq1$ is odd and $1< c < 2$;

\item[III.2.] Phase portrait $(iii-2)$ of Figure \ref{gfp} if $n=1$, $k\geq1$ is odd and $c \geq 2$;

\item[IV.1.] Phase portrait $(iv-1)$ of Figure \ref{gfp} if $n=1$, $k\geq1$ is even and $0<c<2$;

\item[IV.2.] Phase portrait $(iv-2)$ of Figure \ref{gfp} if $n=1$, $k\geq1$ is even and $c \geq 2$;

\item[V.1.] Phase portrait $(v-1)$ of Figure \ref{gfp} if $n=2$, $k=1$ and $0<c<2$;

\item[V.2.] Phase portrait $(v-2)$ of Figure \ref{gfp} if $n=2$, $k=1$ and $c \geq 2$;

\item[VI.1.] Phase portrait $(vi-1)$ of Figure \ref{gfp} if $n=2$, $k>1$ is even and $0<c<2$;

\item[VI.2.] Phase portrait $(vi-2)$ of Figure \ref{gfp} if $n=2$, $k>1$ is even and $c \geq 2$;

\item[VII.1] Phase portrait $(vii-1)$ of Figure \ref{gfp} if $n=2$, $k>1$ is odd and $0<c<2$;

\item[VII.2] Phase portrait $(vii-2)$ of Figure \ref{gfp} if $n=2$, $k>1$ is odd and $c \geq 2$.
\end{description}
\end{theo}

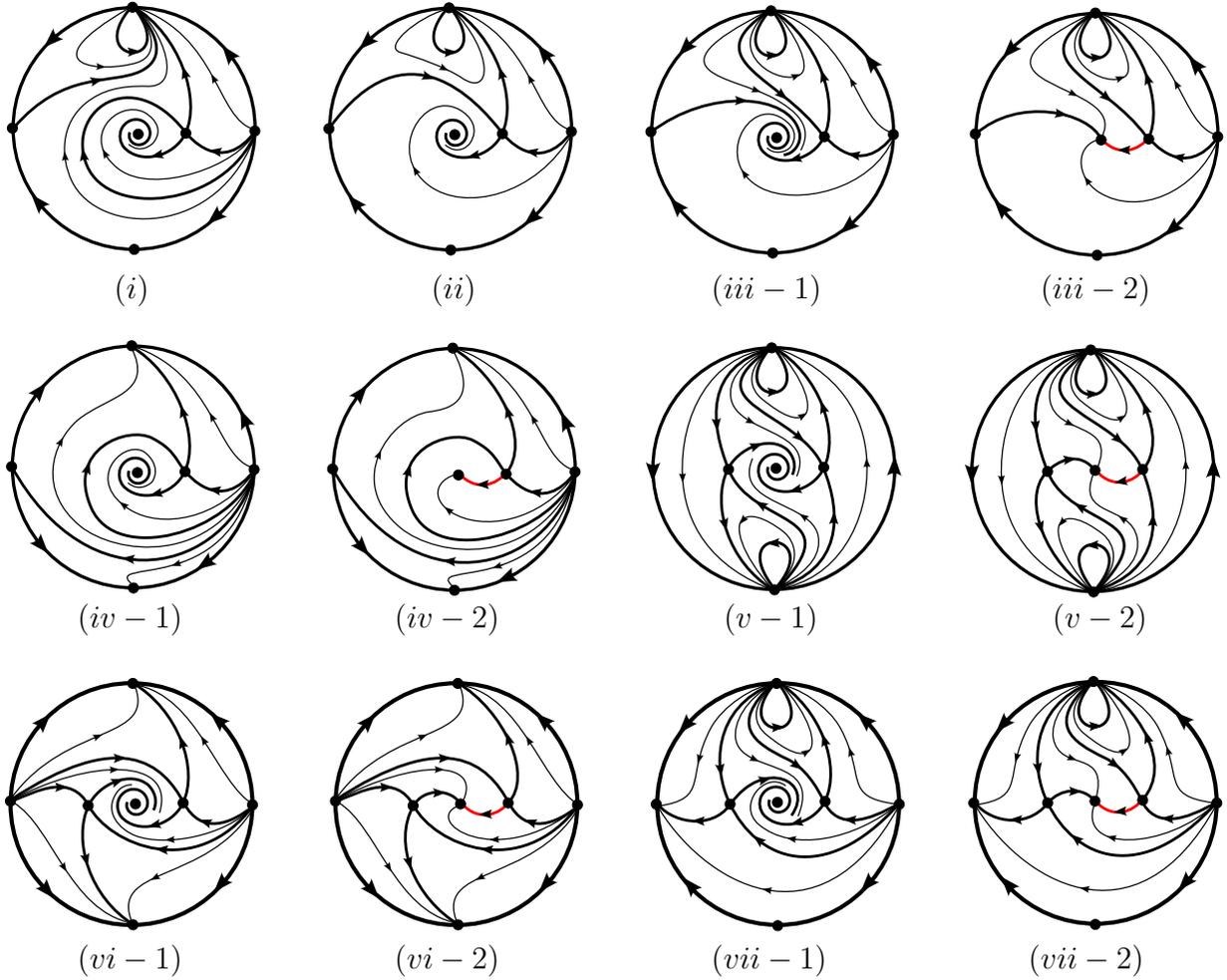
\begin{figure}[h!]
	\begin{center}
		\begin{overpic}[width=6.4in]{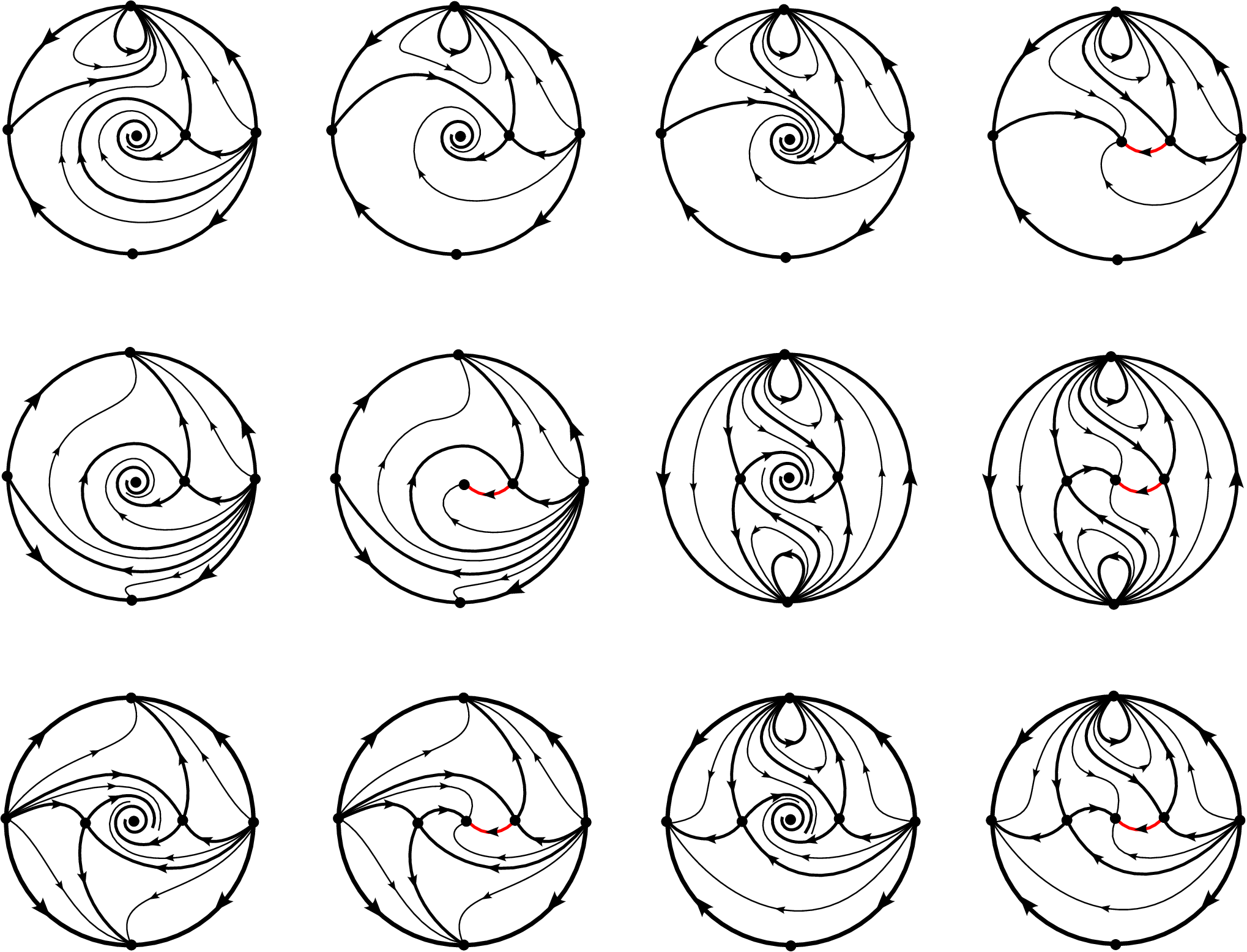}
						\put(9,52) {$(i)$}
						\put(35,52) {$(ii)$}
						\put(58,52) {$(iii-1)$}
						\put(85,52) {$(iii-2)$}
						\put(6,25) {$(iv-1)$}
						\put(32,25) {$(iv-2)$}
						\put(59,25) {$(v-1)$}
						\put(86,25) {$(v-2)$}
						\put(6,-3) {$(vi-1)$}
						\put(32,-3) {$(vi-2)$}
						\put(58,-3) {$(vii-1)$}
						\put(84,-3) {$(vii-2)$}
		\end{overpic}
	\end{center}
	\caption{Global phase portraits of systems \eqref{maineq}. See Theorem \ref{thm:01}. The orbits in red give rise to the traveling wave solutions of \eqref{eq1}. See Theorem \ref{thm:traveling}.}
\label{gfp}
\end{figure}

It is worth mentioning that each of the phase portraits $(iii-1)$, $(iv-1)$, $(v-1)$, $(vi-1)$ and $(vii-1)$ of Figure \ref{gfp} is topologically equivalent to the respective phase portrait $(iii-2)$, $(iv-2)$, $(v-2)$, $(vi-2)$ and $(vii-2)$ of Figure \ref{gfp}. In the phase portraits $(iii-2)$, $(iv-2)$, $(v-2)$, $(vi-2)$ and $(vii-2)$ of Figure \ref{gfp} are illustrated in red the orbits that give rise to the traveling wave solutions of \eqref{eq1}, according to the following theorem.

\begin{theo}\label{thm:traveling}
Consider the partial differential equation \eqref{eq1} under the assumptions: $m = 1$, $n\in\{1,2\}$, $k \in \mathbb N$ and $z \in [0, 1]$. Then, for any given $c \geq 2$, $n\in\{1,2\}$, and $k \in \mathbb N$, equation \eqref{eq1} has a traveling wave solution $\phi(\xi)$ corresponding to the solution of \eqref{eq:n3} connecting the equilibrium points $(1,0)$ and $(0,0)$. See the orbits drawn in red in Figure \ref{gfp}. The solution $\phi(\xi)$ satisfies:
\[
\lim_{\xi \to - \infty} \phi(\xi) = 1, \quad \lim_{\xi \to + \infty} \phi(\xi) = 0, \quad \lim_{\xi \to - \infty} \phi'(\xi) = \lim_{\xi \to + \infty} \phi'(\xi) = 0.
\]
\end{theo}

The paper is organized as follows. In Section \ref{sec:2} we study the global phase portraits of systems \eqref{maineq} using the Poincar\'e compactification and a lot of blow-ups. The proof of Theorem \ref{thm:01} is given in Section \ref{sec:3} while the proof of Theorem \ref{thm:traveling} is given in Section \ref{sec:4}.

\section{Global phase portraits of \eqref{maineq}}\label{sec:2}

The finite equilibrium points are obtained by setting $\dot{x}=0$ and $\dot{y}=0$ in systems \eqref{maineq}. Thus, if $n$ is even, there are three finite equilibrium points: $E_0=(0,0)$, $E_1=(1,0)$ and $E_2=(-1,0)$. Now, if $n$ is odd, there are two finite equilibrium points: $E_0=(0,0)$ and $E_1=(1,0)$. The Jacobian matrix of the vector fields associated to systems \eqref{maineq} is written as
\begin{eqnarray*}
	J(x,y)=\left(\begin{array}{cc}
		0 & 1 \\
		(1 + n) x^n + k x^{k-1} y - 1\,\,\,\, & x^k-c
	\end{array}\right).
\end{eqnarray*}
So, we have that
\begin{eqnarray*}
	J(0,0)=\left(\begin{array}{cc}
		0 & 1 \\
		-1 & -c
	\end{array}\right),\quad J(1,0)=\left(\begin{array}{cc}
	0 & 1 \\
	n & 1-c
	\end{array}\right)\quad\text{and}\quad J(-1,0)=\left(\begin{array}{cc}
	0 & 1 \\
	n & (-1)^k-c
	\end{array}\right).
\end{eqnarray*}
From the Jacobian matrices, we can conclude that $E_0$ is a hyperbolic stable focus if $0<c<2$ or a hyperbolic stable node if $c\geq2$, $E_1$ and $E_2$ are hyperbolic saddles.

\begin{lem}\label{lemma:0}
If $c \geq 1$, then systems \eqref{maineq} do not have closed orbits.
\end{lem}

\begin{proof}
The divergence of the vector fields
\[
F(x, y) = \big( y, \left( x^k - c \right) y + x \left( x^n -1 \right) \big)
\]
 that define systems \eqref{maineq} is given by
\[
\mbox{div}(F(x,y)) = x^k - c, \quad c > 0.
\]
Define the sets
\[
B_1 = \left\{ (x, y) \in \mathbb R^2 : x < c^{1/k} \right\}, \quad B_2 = \left\{ (x, y) \in \mathbb R^2 : -c^{1/k} < x < c^{1/k} \right\}.
\]
It follows that: If $k$ is odd, then $\mbox{div}(F(x,y)) < 0$ for $(x, y) \in B_1$; if $k$ is even, then $\mbox{div}(F(x,y)) < 0$ for $(x, y) \in B_2$. Note that $E_0 = (0,0) \in B_i$, $i = 1, 2$. By the Bendixson's Theorem \cite[Theorem 7.10]{DLA}, systems \eqref{maineq} do not have a closed orbit which lies entirely in $B_i$, $i = 1, 2$. So, if there exists a closed orbit of systems \eqref{maineq} it must cross the boundary of $B_i$, $i = 1, 2$. But this is not possible by analyzing the restrictions of the vector fields $F$ on the $x$-axis and on the boundary of $B_i$, $i = 1, 2$. The lemma is proved.
\end{proof}

Proving the existence or non-existence of closed orbits for systems \eqref{maineq} when $0<c<1$ is not as simple as in Lemma \ref{lemma:0}. However, numerical simulations suggest that these systems do not present closed orbits if $0<c<1$. Thus, the study of the global phase portrait of systems \eqref{maineq} for $0<c<1$ will be made without considering the possible existence of limit cycles.

Since systems \eqref{maineq} are polynomial, one way to study their global dynamics is by studying the Poincar\'e compactification. See \cite[Chapter 5]{DLA}. Firstly we study the infinite equilibrium points using the local charts $U_1$ and $U_2$. This will be done by dividing the study into two cases: $n=1$ and $n=2$.

%
%
%
%
%
%
%

\begin{lem}\label{l1}
Consider systems \eqref{maineq} with $n = 1$.
\begin{description}
\item [(a)] If $k \geq 1$ is odd, then there are four infinite equilibrium points at the boundary of the Poincar\'e disk:
\begin{description}
\item [(a.1)] Two infinite equilibrium points of type saddle-node determined by the local charts $U_1$ and $V_1$;
			
\item [(a.2)] Two infinite equilibrium points each with one hyperbolic sector and one elliptic sector determined by the local charts $U_2$ and $V_2$.
\end{description}
	
\item [(b)] If $k \geq 1$ is even, then there are four infinite equilibrium points at the boundary of the Poincar\'e disk:
\begin{description}
\item [(b.1)] Two infinite equilibrium points of type saddle-node determined by the local charts $U_1$ and $V_1$;
			
\item [(b.2)] Two infinite equilibrium points each with one parabolic sector determined by the local charts $U_2$ and $V_2$.
\end{description}
\end{description}
\end{lem}

\begin{proof}
Since the degree of systems \eqref{maineq} is $k+1$, with $k\geq1$, the stability of the infinite equilibrium points on the local charts $V_1$ and $V_2$ follows directly (resp. opposite) from the study on the local charts $U_1$ and $U_2$ if $k$ is even (resp. odd).

\medskip

\noindent In the local chart $U_1$, systems \eqref{maineq} can be expressed as
\begin{equation}\label{U1c1}
\dot{u} = u + v^{(k-1)} - v^k - c u v^k - u^2 v^k, \quad \dot{v} = -u v^{(k+1)}.
\end{equation}
Setting $v = 0$, if $k=1$ then the only infinite equilibrium point covered by this chart is $I_1=(-1,0)$. Now, if  $k>1$ then the only infinite equilibrium point covered by this chart is ${\tilde I}_1=(0,0)$. The Jacobian matrices associated to systems \eqref{U1c1} at the equilibrium points $I_1$ and ${\tilde I}_1$ are expressed as
\begin{eqnarray*}
	J(-1,0)=\left(\begin{array}{cc}
		1 & c-2 \\
		0 & 0
	\end{array}\right)\text{\quad and \quad} J(0,0)=\left(\begin{array}{cc}
	1 & 0 \\
	0 & 0
	\end{array}\right).
\end{eqnarray*}

\noindent {\bf Case $k=1$.} The point $I_1$ is semi-hyperbolic and require a more detailed analysis. Semi-hyperbolic equilibrium points can be classified using classical theorems from the literature, for instance \cite[Theorem 2.19]{DLA}. After a translation and a linear change of variables, the semi-hyperbolic equilibrium point can be studied as the origin of systems of the form
\begin{equation*}\label{eq:semihyp}
	\dot{u}=A_1(u,v),\quad \dot{v}=v+B_1(u,v),
\end{equation*}
with $A_1$ and $B_1$ having neither constant nor linear terms in its power series expansion. Considering the function $v=f_1(u)$ as the solution of $\dot{v}=0$. If $c\ne1$ and $c\ne2$, then $f_1(u)\ne0$. If $c=1$ or $c=2$, then $f_1(u)=0$. In both cases, $A_1(u,f_1(u))=u^2+O(u^3)$. By \cite[Theorem 2.19]{DLA}, the equilibrium point $I_1$ is a saddle-node (see Figure \ref{fign=k=1} $(a)$). The same occurs with the infinite equilibrium point $I_2$ in the local chart $V_1$, with the opposite stability.

\medskip

\noindent {\bf Case $k>1$.} For this case, the stability analysis of the equilibrium point ${\tilde I}_1$ is similar to the previous case, and the details will be omitted. For this case, it is worth highlighting that $A_1(u,f_1(u))=u^{2k}+O(u^{2k+1})$, that is, by \cite[Theorem 2.19]{DLA}, the equilibrium point ${\tilde I}_1$ is a saddle-node (see Figure \ref{fign=k=1} $(a)$). The same occurs with the infinite equilibrium point $I_2$ in the local chart $V_1$, with the same stability if $k>1$ is even and opposite stability if $k>1$ is odd. Statements (a.1) and (b.1) are proved.


\noindent In the local chart $U_2$, systems \eqref{maineq} write as
\begin{equation}\label{U2c1}
	\dot{u} = -u^3 v^{k-1} + v^k + c u v^k + u^2 v^k - u^{k + 1}, \quad \dot{v} =-u^2 v^{k} + c v^{k+1} + u v^{k+1} - u^k v.
\end{equation}
The only equilibrium point which is not covered by the local chart $U_1$
is $I_3 = (0, 0)$. The Jacobian matrices associated to systems \eqref{U2c1} at the equilibrium point $I_3$ are expressed as
\begin{eqnarray*}
	J(0,0)=\left(\begin{array}{cc}
		0 & 1 \\
		0 & 0
	\end{array}\right),\text{\quad if $k=1$,\quad or\quad} J(0,0)=\left(\begin{array}{cc}
	0 & 0 \\
	0 & 0
	\end{array}\right),\text{\quad if $k>1$}.
\end{eqnarray*}

\noindent {\bf Case $k=1$.} In this case, the point $I_3$ is nilpotent. Nilpotent equilibrium points can be classified using theorems from the literature, see \cite[Theorem 3.5]{DLA}. Considering the functions $v=f_0(u)$ as the solution of $\dot{u}=0$, we obtain $B_0(u,f_0(u))=-u^3+O(u^4)$ and $(\partial A_0/\partial u+\partial A_0/\partial v)(u,f_0(u))=-3 u+O(u^2)$. By \cite[Theorem 3.5]{DLA}, there are one hyperbolic and one elliptic sector in a neighborhood of the equilibrium point $I_3$ (see Figure \ref{fign=k=1} $(b)$). The same occurs with the infinite equilibrium point $I_4$ in the local chart $V_2$, with the opposite stability.

\begin{figure}[h]
	\begin{center}
		\begin{overpic}[width=4.5in]{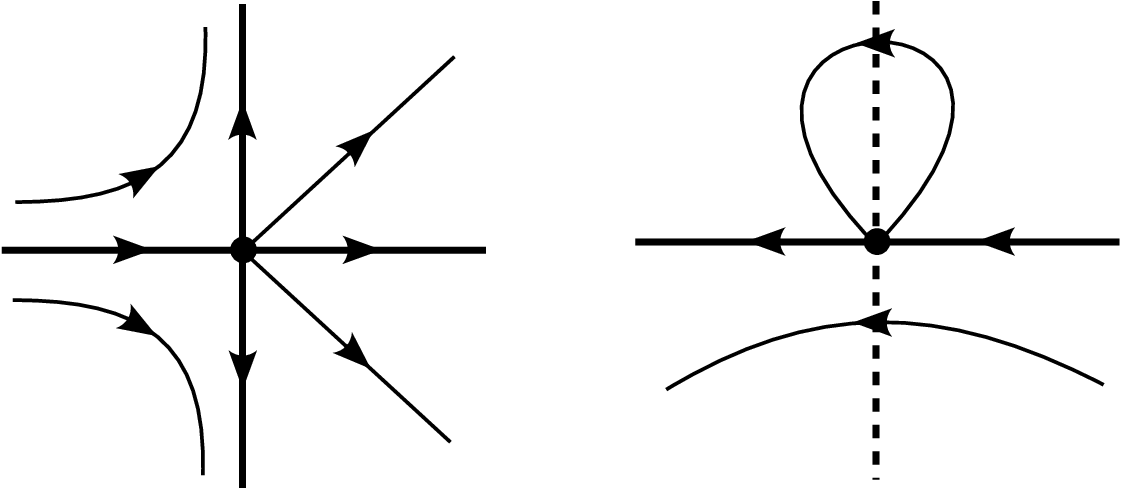}
			\put(20,-4) {$(a)$}
			\put(76,-4) {$(b)$}
		\end{overpic}
	\end{center}
	\caption{$(a)$ Topological local phase portraits at $I_1$ and ${\tilde I}_1$ of systems \eqref{U1c1}. $(b)$ Topological local phase portraits at $I_3$ of systems \eqref{U2c1} if $k=1$.}
	\label{fign=k=1}
\end{figure}

\medskip

\noindent {\bf Case $k>1$.} This case requires more attention since $I_3$ is a degenerate equilibrium point. We use the quasi-homogeneous polar blow-ups to describe the local dynamics at this point. See \cite{AFJ} and \cite{DLA}. Consider the blow-up given by $(u, v)=(r^k \cos(\theta), r^{k+1} \sin(\theta))$, with $r \geq 0$ and $0 \leq \theta <2\pi$. After eliminating the common factor $(k\cos^2(\theta)+(k+1)\sin^2(\theta))/r^{k^2}$, systems \eqref{U2c1} have the form
\begin{equation}\label{buc1}
	\begin{split}
		&\dot{r}=
		r \left( \cos(\theta) \sin(\theta)^k - \cos(\theta)^k \right) - r^{2k} \cos(\theta)^2 \sin(\theta)^{k-1} + c\, r^{k+1} \sin(\theta)^k \\
		&\quad\,\,+ r^{2k+1} \cos(\theta) \sin(\theta)^k,
		\\[0.25cm]
		&\dot{\theta}= \sin(\theta) \left( \cos(\theta)^{k+1} - (1+k) \sin(\theta)^k \right) + r^{2k-1} \cos(\theta)^3 \sin(\theta)^k - c r^k \cos(\theta) \sin(\theta)^{k+1} \\
		&\quad\,\,- r^{2k} \cos(\theta)^2 \sin(\theta)^{k+1}.
	\end{split}
\end{equation}
The equilibrium point $I_3$ is blown up into the circle $r = 0$. The equilibrium points
belonging to this circle  are
\begin{gather*}
	\theta_0=0,\quad\theta_1\in \left(0,\dfrac{\pi}{2}\right),\quad\theta_2=\pi,\quad \theta_3\in
	\begin{cases}
		\left(\dfrac{\pi}{2},\pi\right)\,\,\text{if $k$ is odd}\vspace{0.2cm},\\
		\left(\dfrac{3\pi}{2},2\pi\right)\,\,\text{if $k$ is even,}
	\end{cases}
\end{gather*}
which are solutions of the equation
\begin{equation}\label{eqc1}
	\sin(\theta) \left(\cos(\theta)^{k+1} - (1 + k) \sin(\theta)^k\right)=0.
\end{equation}
The Jacobian matrix associated to systems \eqref{buc1} at the equilibrium points $(0, \theta_i)$, for $i=0,1,2,3$, is expressed as
\begin{eqnarray*}
	J(0,\theta_i)=\left(\begin{array}{cc}
		\cos(\theta_i)\sin^k(\theta_i)-\cos^k(\theta_i) & 0 \\
		0 & \alpha
	\end{array}\right),
\end{eqnarray*}
where $\alpha=\cos^{k+2}(\theta_i)-(k+1)\cos^k(\theta_i)\sin^2(\theta_i)-(k+1)^2\cos(\theta_i)\sin^k(\theta_i)$. It follows that the Jacobian matrices at the equilibrium points $(0, \theta_0)$ and $(0, \theta_2)$ can be written as
\begin{gather*}
	J(0,\theta_0)=\left(\begin{array}{cc}
		-1 & 0 \\
		0 & 1
	\end{array}\right)\text{\quad and \quad}
	J(0,\theta_2)=\left(\begin{array}{cc}
		(-1)^{k+1} & 0 \\
		0 & (-1)^{k}
	\end{array}\right).
\end{gather*}
The equilibrium point $(0,\theta_0)$ is a hyperbolic saddle attracting in the radial direction and repelling in the polar direction. The equilibrium point $(0,\theta_2)$ is a hyperbolic saddle attracting (resp. repelling) in the radial direction and repelling (resp. attracting) in the polar direction if $k$ is even (resp. odd).

\noindent In order to study the stability of the equilibrium points  $(0,\theta_1)$ and  $(0,\theta_3)$ note that $\sin^k(\theta)=\cos^{k+1}(\theta)/(k+1)$, see \eqref{eqc1}. It follows that the Jacobian matrix associated to systems \eqref{buc1} at the equilibrium points $(0, \theta_i)$, for $i=1,3$, can be expressed as
\begin{eqnarray*}
	J(0,\theta_i)=\left(\begin{array}{cc}
		\dfrac{\cos^k(\theta_i)\,\alpha_i}{k+1} & 0 \vspace{0.2cm}\\
		0 & \cos^k(\theta_i)\,\alpha_i
	\end{array}\right),
\end{eqnarray*}
where $\alpha_i=\cos^2(\theta_i)-(k+1)<0$. Thus, the equilibrium point $(0,\theta_1)$ is a hyperbolic stable node or focus for $\theta_1\in(0,\pi/2)$ and the equilibrium point $(0,\theta_3)$ is a hyperbolic stable (resp. unstable) node or focus for $\theta_3\in(3\pi/2,2\pi)$ (resp. $\theta_3\in(\pi/2,\pi)$).

\medskip

\noindent A summary of the study of the dynamics at $r=0$ is illustrated in Figures \ref{fign1k>2} $(a)$ and $(d)$. Going back through the changes of variables, the dynamics close to the circle $r=0$ allow us to classify the infinite equilibrium point $I_3$ of systems \eqref{U2c1}. It follows that if $k$ is odd then the equilibrium point has one elliptic sector and one hyperbolic sector (see Figure \ref{fign1k>2} $(b)$). We will denote this point by $I_3^{0}$. The same occurs with the infinite equilibrium point $I_4^{0}$ in the local chart $V_2$, with the opposite stability. Statement (a.2) is proved. Now, if $k$ is even then the equilibrium point has one parabolic sector (see Figure \ref{fign1k>2} $(e)$). We will denote this point by $I_3^{e}$. The same occurs with the infinite equilibrium point $I_4^{e}$ in the local chart $V_2$, with the same stability. Statement (b.2) is proved. The local behaviors of the infinite points studied in this subsection are depicted in Figures \ref{fign1k>2} $(c)$ and $(f)$.
\end{proof}

\begin{figure}[h!]
	\begin{center}
		\begin{overpic}[width=5.5in]{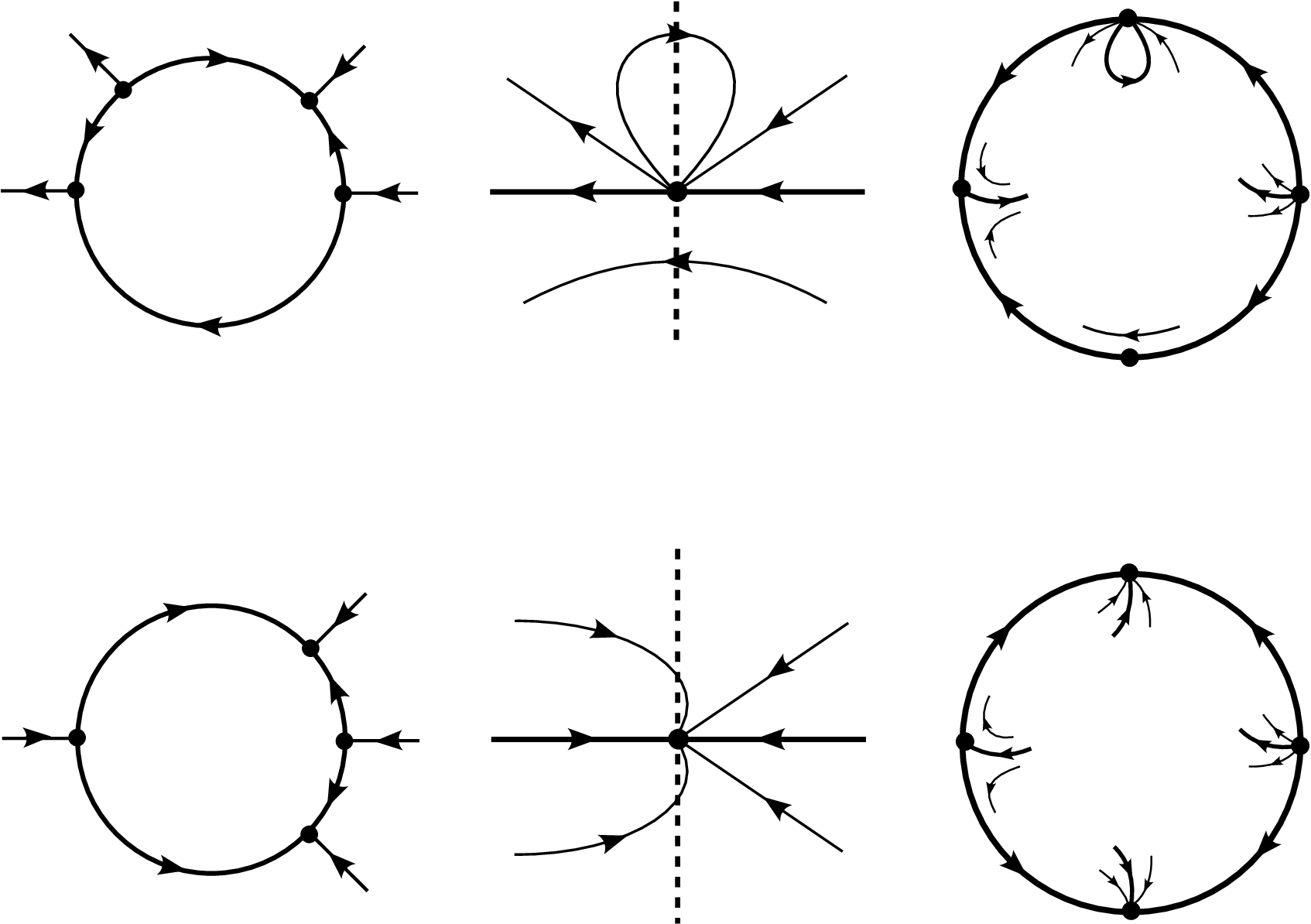}
			\put(15,35) {$(a)$}
			\put(50,35) {$(b)$}
			\put(85,35) {$(c)$}
			\put(15,-7) {$(d)$}
			\put(50,-7) {$(e)$}
			\put(85,-7) {$(f)$}
			\put(100,55) {$I_1$}
			\put(70,55) {$I_2$}
			\put(85,71) {$I_3^0$}
			\put(85,39.5) {$I_4^0$}
			\put(100,12) {$I_1$}
			\put(70,13) {$I_2$}
			\put(85,29) {$I_3^e$}
			\put(85,-3) {$I_4^e$}
			\put(32,55) {$\theta_0$}
			\put(-2.5,55) {$\theta_2$}
			\put(28,67) {$\theta_1$}
			\put(2.5,68) {$\theta_3$}
			\put(32,13) {$\theta_0$}
			\put(-2.5,13) {$\theta_2$}
			\put(28,25) {$\theta_1$}
			\put(28,1) {$\theta_3$}
			\put(65,54) {$u$}
			\put(52,70) {$v$}
			\put(65,12) {$u$}
			\put(52,28) {$v$}
		\end{overpic}
	\end{center}
\vspace{0.5cm}
\caption{$(a)$ Desingularization of systems \eqref{U2c1} with $k>1$ odd using polar blow-ups. $(b)$ Topological local phase portraits at the origin of systems \eqref{U2c1} with $k>1$ odd. $(c)$ Dynamics close to the infinity of systems \eqref{maineq} with $n=1$ and $k\geq1$ odd in the Poincar\'e disk. $(d)$ Desingularization of systems \eqref{U2c1} using polar blow-ups whit $k>1$ even. $(e)$ Topological local phase portraits at the origin of systems \eqref{U2c1} with $k>1$ even. $(f)$ Dynamics close to the infinity of systems \eqref{maineq} with $n=1$ and $k\geq1$ even in the Poincar\'e disk.}
\label{fign1k>2}
\end{figure}
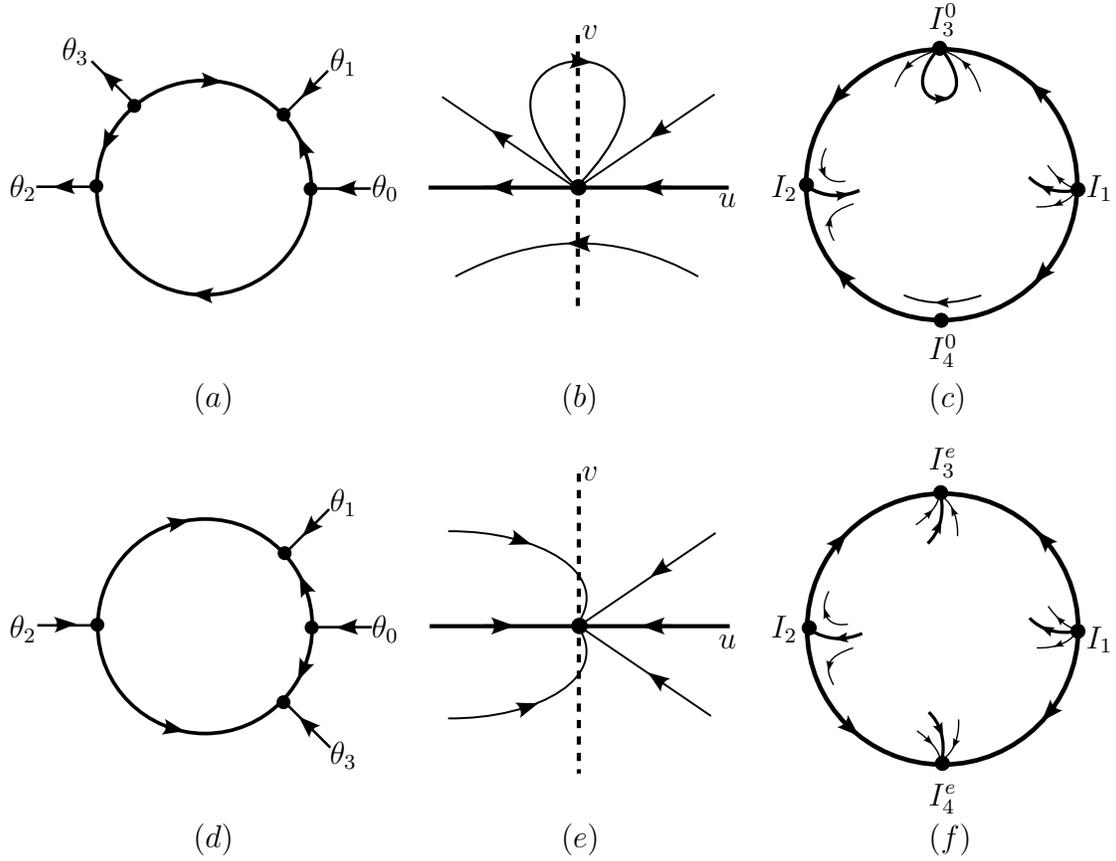

\begin{lem}\label{l2}
Consider systems \eqref{maineq} with $n = 2$.
\begin{description}
\item [(a)] If $k = 1$, then there are two infinite equilibrium points at the boundary of the Poincar\'e disk each with two elliptic sectors and two parabolic sectors determined by the local charts $U_2$ and $V_2$.
		
\item [(b)] If $k > 1$ is even, then there are four infinite equilibrium points at the boundary of the Poincar\'e disk each with one parabolic sector determined by the local charts $U_1$, $V_1$, $U_2$ and $V_2$;
		
\item [(c)] If $k > 1$ is odd, then there are four infinite equilibrium points at the boundary of the Poincar\'e disk:
\begin{description}
\item [(c.1)] Two infinite equilibrium points each with one parabolic sector determined by the local charts $U_1$ and $V_1$;
			
\item [(c.2)] Two infinite equilibrium points each with one elliptic sector and one hyperbolic sector determined by the local charts $U_2$ and $V_2$.
\end{description}
\end{description}
\end{lem}

\begin{proof}
Since systems \eqref{maineq} are cubic if $k\in\{1,2\}$ and has degree $k+1$ if $k\geq3$, the stability of the infinite equilibrium points on the local charts $V_1$ and $V_2$ follows directly (resp. opposite) from the study on the local charts $U_1$ and $U_2$ if  $k\in\{1,2\}$ or $k\geq3$ is even (resp. $k\geq3$ is odd).

\medskip

\noindent To simplify the study of the stability of the equilibrium points in the local chart $U_1$, if they exist, we will divide it into two cases: $k = 1$ and $k > 1$.

\medskip

\noindent {\bf Case $k=1$.} In the local chart $U_1$, systems \eqref{maineq} with $k=1$ can be written as
\begin{equation}\label{U1n2c1}
	\dot{u} = 1 + v (u - (1 + u (c + u)) v), \quad \dot{v} = -u v^3.
\end{equation}
Taking $v = 0$, it follows that systems \eqref{U1n2c1} has no equilibrium points.

\medskip

\noindent {\bf Case $k>1$.} In the local chart $U_1$, systems \eqref{maineq} with $k>1$ can be expressed as
\begin{equation}\label{U1n2c2}
	\dot{u} = u + v^{(k-2)} - v^k - c u v^k - u^2 v^k, \quad \dot{v} = -u v^{(k+1)}.
\end{equation}
Consider $v = 0$. If $k=2$ then the only infinite equilibrium point covered by this chart is $I_5=(-1,0)$. If $k>2$ then the only infinite equilibrium point covered by this chart is ${\tilde I}_5=(0,0)$. The Jacobian matrix associated to systems \eqref{U1n2c2} at the equilibrium points $I_5$ and ${\tilde I}_5$ is given by
\begin{eqnarray*}
	J(0,0)=J(-1,0)=\left(\begin{array}{cc}
		1 & \delta \\
		0 & 0
	\end{array}\right),
\end{eqnarray*}
where $\delta=1$ if $k=3$, and $\delta=0$ otherwise. The points $I_5$ and ${\tilde I}_5$ are semi-hyperbolic. Using the same techniques presented in the previous cases, it is possible to show that both points are unstable nodes. The same occurs with the infinite equilibrium point $I_6$ in the local chart $V_1$, with the same stability if $k>1$ is even and the opposite stability if $k>1$ is odd. Statement (c.1) is proved.

\medskip

\noindent Once again, to simplify the study in the local chart $U_2$, we will divide it into two cases: $k = 1$ and $k > 1$.

\medskip

\noindent {\bf Case $k=1$.} In the local chart $U_2$, systems \eqref{maineq} with $k=1$ have the form
\begin{equation}\label{U2c2}
	\dot{u} = -u^4 + u^2 (v-1) v + v^2 (1+c u), \quad \dot{v} = v (-u^3 + u (v-1) v + c v^2).
\end{equation}
Since that $I_7=(0,0)$ is a degenerate equilibrium point, we use the quasi-homogeneous polar blow-ups to describe the local dynamics at this point. See \cite{AFJ} and \cite{DLA}. Consider the blow-up given by $(u, v)=(r \cos(\theta), r^2 \sin(\theta))$, with $r \geq 0$ and $0 \leq \theta <2\pi$. Applying this blow-up and eliminating the common factor $(1+\sin^2(\theta))/r^3$, systems \eqref{U2c2} are written as
\begin{equation}\label{buc2}
	\begin{split}
		&\dot{r}=
		\frac{r}{4} \left( 2\, c\, r + (-2 + r^2) \cos(\theta) - 2\, c\, r \cos(2\theta)
		- (2 + r^2) \cos(3\theta) - 2 \sin(2\theta) \right),
		\\[0.25cm]
		&\dot{\theta}= \sin(\theta) \left( \cos^4(\theta) + \cos^2(\theta) \sin(\theta) - 2 \sin^2(\theta) \right)-r \cos(\theta) \left( c + r \cos(\theta) \right) \sin^3(\theta).
	\end{split}
\end{equation}
The equilibrium point $I_7$ is blown up into the circle $r = 0$. The equilibrium points belonging to this circle  are
\begin{gather*}
\theta_0^1=0,\quad\theta_1^1=\arcsin{\left(\frac{\sqrt{5}-1}{2}\right)},\quad\theta_2^1=\pi-\arcsin{\left(\frac{\sqrt{5}-1}{2}\right)},\quad \theta_3^1=\pi,\\
	\hspace{-1cm}\theta_4^1=\pi-\arcsin{\left(1-\sqrt{2}\right)}\quad\text{and}\quad
	\theta_5^1=2\pi+\arcsin{\left(1-\sqrt{2}\right)},
\end{gather*}
which are solutions of the equation
\begin{equation*}
	\sin(\theta) \left( \cos^4(\theta) + \cos^2(\theta) \sin(\theta) - 2 \sin^2(\theta) \right)=0.
\end{equation*}
The Jacobian matrices associated to systems \eqref{buc2} at these equilibrium points are the following
\begin{gather*}
	J(0,\theta_0^1)=-J(0,\theta_3^1)=\left(\begin{array}{cc}
		-1 & 0 \\
		0 & 1
	\end{array}\right), \\[0.25cm]
	J(0,\theta_1^1)=-J(0,\theta_2^1)=\left(\begin{array}{cc}
		\frac{1}{2} \left( 5 - 3 \sqrt{5} \right) \sqrt{\frac{1}{2} \left( -1 + \sqrt{5} \right)} & 0 \vspace{0.2cm}\\
		-\dfrac{\left( -1 + \sqrt{5} \right)^{\frac{7}{2}} c}{8 \sqrt{2}} & -\dfrac{3}{2} \sqrt{\frac{1}{2} \left( -1 + \sqrt{5} \right)} \left( -5 + 3 \sqrt{5} \right)
	\end{array}\right),\\[0.25cm]
	J(0,\theta_4^1)=-J(0,\theta_5^1)=\left(\begin{array}{cc}
		\dfrac{2}{\sqrt{41 + 29 \sqrt{2}}} & 0 \vspace{0.2cm}\\
		-\sqrt{2} \left( -1 + \sqrt{2} \right)^{\frac{7}{2}} c & \dfrac{12}{\sqrt{41 + 29 \sqrt{2}}}
	\end{array}\right).
\end{gather*}
It follows that $(0,\theta_0^1)$ and $(0,\theta_3^1)$ are hyperbolic saddles. More precisely, $(0,\theta_0^1)$ (resp. $(0,\theta_3^1)$) is attracting (resp. repelling) in the radial direction ($r$-direction) and repelling (resp. attracting) in the angular direction ($\theta$-direction). Moreover, $(0,\theta_1^1)$ and $(0,\theta_5^1)$ are hyperbolic stable node or focus,  $(0,\theta_2^1)$ and $(0,\theta_4^1)$ are hyperbolic unstable node or focus.

\medskip

\noindent A summary of the study of the dynamics at $r=0$ is illustrated in Figure \ref{fign=2k=1} $(a)$. Going back through the changes of variables, the dynamics close to the circle $r=0$ allow us to classify the infinite equilibrium point $I_7$ of systems \eqref{U2c2}. See Figure \ref{fign=2k=1} $(b)$. It follows that the equilibrium point has two elliptic sectors and two parabolic sectors.  The same stability occurs with the infinite equilibrium point $I_8$ in the local chart $V_2$. Statement (a) is proved. The local behaviors of the infinite equilibrium points of systems \eqref{maineq} with $n=2$ and $k=1$ are depicted in Figure \ref{fign=2k=1} $(c)$.

\begin{figure}[h!]
	\begin{center}
		\begin{overpic}[width=5.5in]{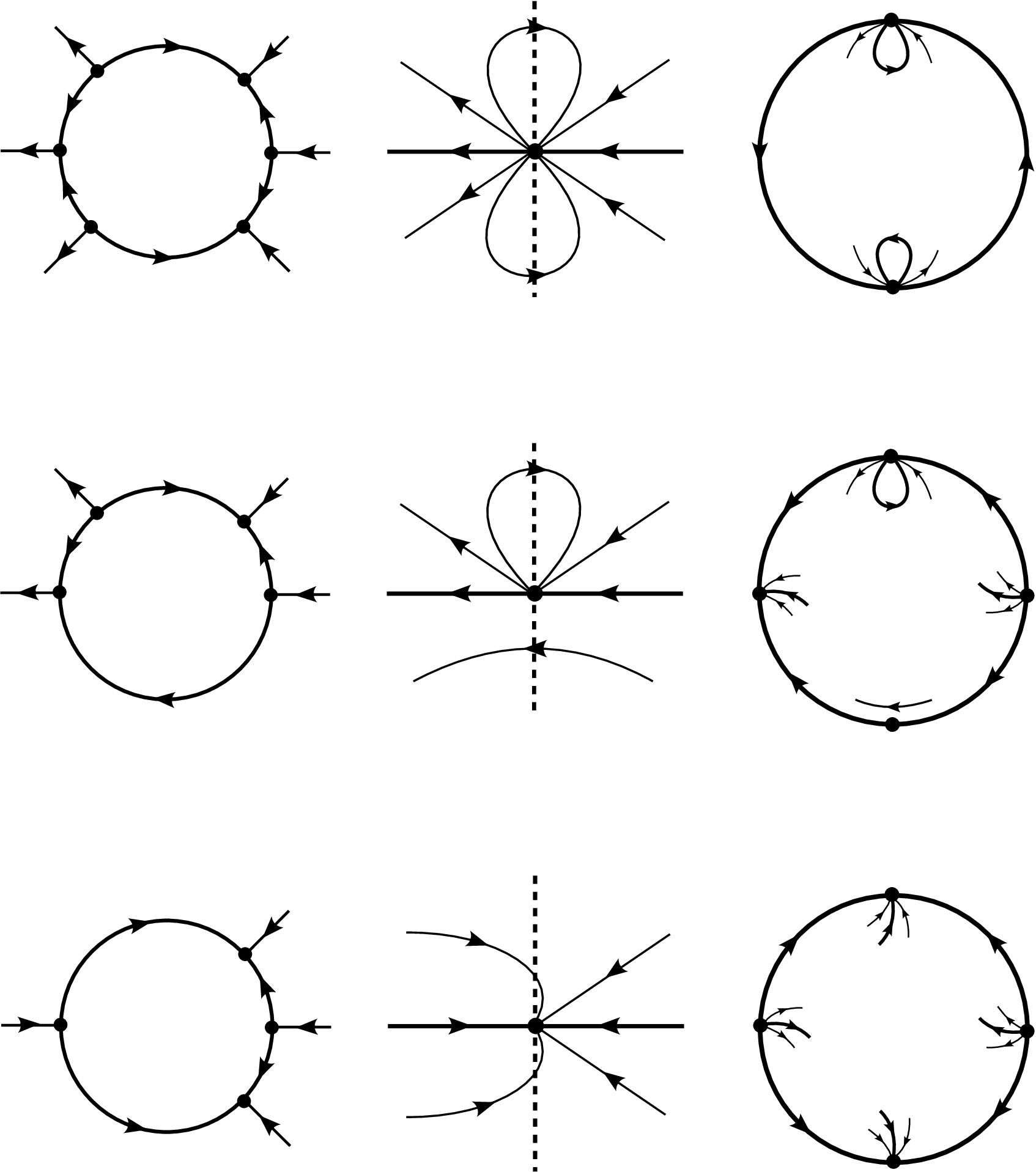}
						\put(12,69) {$(a)$}
						\put(44,69) {$(b)$}
						\put(75,69) {$(c)$}
						\put(12,31) {$(d)$}
						\put(44,31) {$(e)$}
						\put(75,31) {$(f)$}
						\put(12,-6) {$(g)$}
						\put(44,-6) {$(h)$}
						\put(75,-6) {$(i)$}
						\put(89,48) {$I_5$}
						\put(61.5,48) {$I_6$}
						\put(89,11) {$I_5$}
						\put(61.5,12) {$I_6$}
						\put(75,100) {$I_7$}
						\put(75,72.5) {$I_8$}
						\put(75,63) {$I_9^0$}
						\put(75,35) {$I_{10}^0$}
						\put(75,25.5) {$I_9^e$}
						\put(75,-2) {$I_{10}^e$}
						\put(28.5,86) {$\theta_0^1$}
						\put(25,96) {$\theta_1^1$}
						\put(2,97) {$\theta_2^1$}
						\put(-2,86) {$\theta_3^1$}
						\put(1.5,75) {$\theta_4^1$}
						\put(25,75) {$\theta_5^1$}
						\put(28.5,48) {$\theta_0^2$}
						\put(25,59) {$\theta_1^2$}
						\put(-2.5,49) {$\theta_2^2$}
						\put(2,59) {$\theta_3^2$}
						\put(28.5,11) {$\theta_0^2$}
						\put(25,22) {$\theta_1^2$}
						\put(-2,12) {$\theta_2^2$}
						\put(25,1) {$\theta_3^2$}
						\put(57,85) {$u$}
						\put(46,99) {$v$}
						\put(57,47) {$u$}
						\put(46,61) {$v$}
						\put(57,10.5) {$u$}
						\put(46,24) {$v$}
		\end{overpic}
	\end{center}\vspace{0.5cm}
\caption{$(a)$ Desingularization of systems \eqref{U2c2} using polar blow-ups. $(b)$ Topological local phase portraits at the origin of systems \eqref{U2c2}. $(c)$ Dynamics close to the infinity of systems \eqref{maineq} with $n=2$ and $k=1$ in the Poincar\'e disk. $(d)$ Desingularization of systems \eqref{U2c3} with $k>1$ odd using polar blow-ups. $(e)$ Topological local phase portraits at the origin of systems \eqref{U2c3} with $k>1$ odd. $(f)$ Dynamics close to the infinity of systems \eqref{maineq} with $n=2$ and $k>1$ odd in the Poincar\'e disk. $(g)$ Desingularization of systems \eqref{U2c3} with $k>1$ even using polar blow-ups. $(h)$ Topological local phase portraits at the origin of systems \eqref{U2c3} with $k>1$ even. $(i)$ Dynamics close to the infinity of systems \eqref{maineq} with $n=2$ and $k>1$ even in the Poincar\'e disk.}
\label{fign=2k=1}
\end{figure}
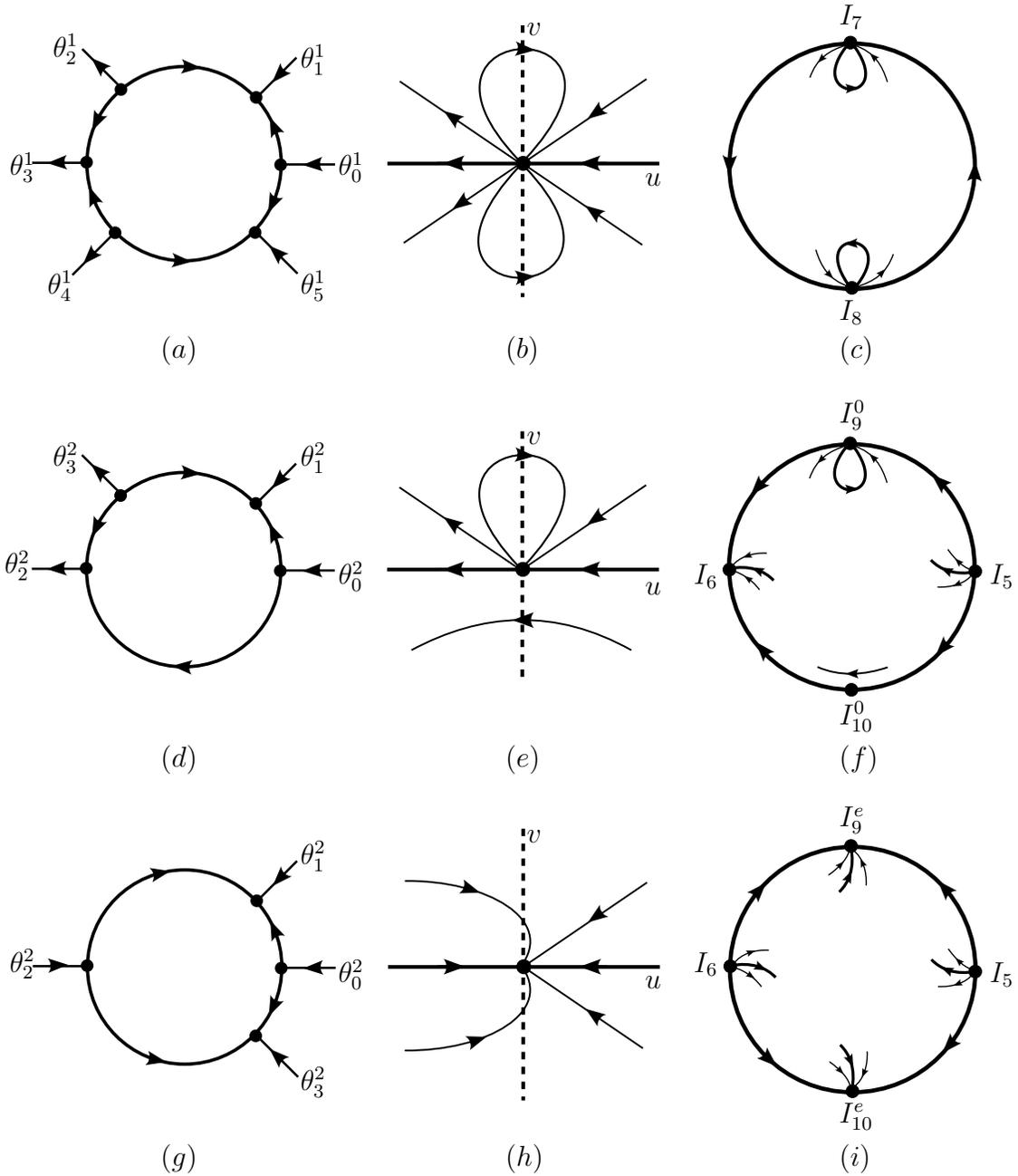

\medskip

\noindent {\bf Case $k>1$.} In the local chart $U_2$, systems \eqref{maineq} with $k>1$ are written as
\begin{equation}\label{U2c3}
	\dot{u} = -u^4 v^{k-2} + v^k + c u v^k + u^2 v^k - u^{k + 1}, \quad \dot{v} =-u^3 v^{k-1} + c v^{k+1} + u v^{k+1} - u^k v.
\end{equation}
Since that $I_9=(0,0)$ is a degenerate equilibrium point, we use the quasi-homogeneous polar blow-ups to describe the local dynamics at this point. The study of the stability of the equilibrium point  $I_9$ is similar to the case $n=1$ and $k>1$, and the details will be omitted.

\noindent Consider the blow-up given by $(u, v)=(r^k \cos(\theta), r^{k+1} \sin(\theta))$, with $r \geq 0$ and $0 \leq \theta <2\pi$. Doing this blow-up and eliminating the common factor $(k\cos^2(\theta)+(k+1)\sin^2(\theta))/r^{k^2}$, systems \eqref{U2c3} have the form
\begin{equation}\label{buc3}
	\begin{split}
		&\dot{r}=
		r \left(\cos(\theta) \sin(\theta)^k - \cos(\theta)^k\right) + r^{2k - 1} \left(\cos(\theta) \sin(\theta)^k - \cos(\theta) \sin(\theta)^{k - 2}\right)  \\
		&\quad\,\,+ c r^{k + 1} \sin(\theta)^k + r^{2k + 1} \cos(\theta) \sin(\theta)^k,
		\\
		&\dot{\theta}= \sin(\theta) \left(\cos(\theta)^{k+1} - (1 + k) \sin(\theta)^k\right) + r^{2k - 2} \left(\cos(\theta)^2 \sin(\theta)^{k - 1} - \cos(\theta)^2 \sin(\theta)^{k + 1}\right) \\[0.25cm]
		&\quad\,\,- r^k \left(c \cos(\theta) \sin(\theta)^{k + 1} + r^k \cos(\theta)^2 \sin(\theta)^{k + 1}\right).
	\end{split}
\end{equation}
The equilibrium point $I_0 = (0, 0)$ is blown up into the circle $r = 0$. The equilibrium points belonging to this circle  are
\begin{gather*}
	\theta_0^2=0,\quad\theta_1^2\in \left(0,\dfrac{\pi}{2}\right),\quad\theta_2^2=\pi,\quad \theta_3^2\in
	\begin{cases}
		\left(\dfrac{\pi}{2},\pi\right)\,\,\text{if $k$ is odd}\vspace{0.2cm},\\
		\left(\dfrac{3\pi}{2},2\pi\right)\,\,\text{if $k$ is even,}
	\end{cases}
\end{gather*}
which are solutions of the equation
\begin{equation}\label{eqc3}
	\sin(\theta) \left(\cos(\theta)^{k+1} - (1 + k) \sin(\theta)^k\right)=0.
\end{equation}
The Jacobian matrices associated to systems \eqref{buc3} at the equilibrium points $(0, \theta_i^2)$, for $i=0,1,2,3$, are given by
\begin{eqnarray*}
	J(0,\theta_i^2)=\left(\begin{array}{cc}
		\cos(\theta_i^2)\sin^k(\theta_i^2)-\cos^k(\theta_i^2) & 0 \\
		0 & \alpha
	\end{array}\right),
\end{eqnarray*}
where $\alpha=\cos^{k+2}(\theta_i^2)-(k+1)\cos^k(\theta_i^2)\sin^2(\theta_i^2)-(k+1)^2\cos(\theta_i^2)\sin^k(\theta_i^2)$. It follows that the Jacobian matrices at the equilibrium points $(0, \theta_0^1)$ and $(0, \theta_2^2)$ are
\begin{gather*}
	J(0,\theta_0^2)=\left(\begin{array}{cc}
		-1 & 0 \\
		0 & 1
	\end{array}\right)\text{\quad and \quad}
	J(0,\theta_2^2)=\left(\begin{array}{cc}
		(-1)^{k+1} & 0 \\
		0 & (-1)^{k}
	\end{array}\right).
\end{gather*}
The equilibrium point $(0,\theta_0^2)$ is a hyperbolic saddle attracting in the radial direction and repelling in the polar direction. The equilibrium point $(0,\theta_2^2)$ is a hyperbolic saddle attracting (repelling) in the radial direction and repelling (attracting) in the polar direction if $k$ is even (odd).

\medskip

\noindent In order to study the stability of the equilibrium points  $(0,\theta_1^2)$ and  $(0,\theta_3^2)$, note that $\sin^k(\theta)=\cos^{k+1}(\theta)/(k+1)$ (see \eqref{eqc3}). It follows that the Jacobian matrices associated to systems \eqref{buc3} at the equilibrium points $(0, \theta_i^2)$, for $i=1,3$, can be expressed as
\begin{eqnarray*}
	J(0,\theta_i^2)=\left(\begin{array}{cc}
		\dfrac{\cos^k(\theta_i^2)\,\alpha_i}{k+1} & 0 \\
		0 & \cos^k(\theta_i^2)\,\alpha_i^2
	\end{array}\right),
\end{eqnarray*}
where $\alpha_i^2=\cos^2(\theta_i^2)-(k+1)<0$. Thus, the equilibrium point $(0,\theta_1^2)$ is a hyperbolic stable node or focus for $\theta_1^2\in(0,\pi/2)$ and the equilibrium point $(0,\theta_3^2)$ is a hyperbolic stable (unstable) node or focus for $\theta_3^2\in(3\pi/2,2\pi)$ ($\theta_3^2\in(\pi/2,\pi)$).

\medskip

\noindent A summary of the study of dynamics at $r=0$ is illustrated in Figures \ref{fign=2k=1} $(d)$ and $(g)$. Going back through the changes of variables, the dynamics close to the circle $r=0$ allow us to classify the infinite equilibrium point $I_9$ of systems \eqref{U2c3}. It follows that if $k$ is odd then the equilibrium  has one elliptic sector and one hyperbolic sector (see Figure \ref{fign=2k=1} $(e)$). Denote this equilibrium point by $I_9^0$. The same occurs with the infinite equilibrium point $I_{10}^0$ in the local chart $V_2$, with appositive stability. Statement (c.2) is proved. Now, if $k$ is even then the equilibrium has one parabolic sector (see Figure \ref{fign=2k=1} $(h)$). Denote this equilibrium point by $I_9^e$. The same occurs with the infinite equilibrium point $I_{10}^e$ in the local chart $V_2$ with the same stability. Statement (b) is proved. The local behaviors of the infinite equilibrium points of systems \eqref{maineq} with $n=2$ and $k>1$ are depicted in Figure \ref{fign=2k=1} $(f)$ and $(i)$.
\end{proof}

\section{Proof of Theorem \ref{thm:01}}\label{sec:3}

The proof of Theorem \ref{thm:01} is a combination of the study of the local phase portraits of finite and infinite equilibrium points, as well as the non-existence of closed orbits for $c \geq 1$. See Lemmas \ref{lemma:0}, \ref{l1} and \ref{l2}. In particular, systems \eqref{maineq} do not have limit cycles if $c \geq 1$. Combining it with the possible separatrices connections between the finite and infinite equilibrium points, we obtain the possible seven global phase portraits of systems \eqref{maineq}  depicted in Figure \ref{gfp}. We will present the explicit proofs for the first three global phase portraits in Figure \ref{gfp}, while the others phase portraits are obtained using similar arguments.

Consider systems \eqref{maineq} with $n \in \{1,2\}$. If $n=1$ there are two finite equilibrium points: $E_0=(0,0)$, which is a hyperbolic stable focus if $0<c<2$ or a hyperbolic stable node if $c\geq2$; and $E_1=(1,0)$, which is a hyperbolic saddle. If $n=2$ there are tree finite equilibrium points: $E_0$, which is a hyperbolic stable focus if $0<c<2$ or a hyperbolic stable node if $c\geq2$; $E_1$ and $E_2$ that are hyperbolic saddles. The connections between the stable and unstable separatrices of the finite equilibrium points $E_1$ and $E_2$ with the separatrices of the infinity equilibrium points on $\mathbb{S}^1$ (see Lemmas \ref{l1} and \ref{l2}), determine the boundaries of the canonical regions in the Poincar\'e disk,  that is, a maximal subset of the Poincar\'e disk where the orbits have the same qualitative behavior. It follows from the Markus–Neumann–Peixoto Theorem (see \cite{M, N, P}) that, to classify the phase portraits in the Poincar\'e disk of a planar polynomial differential system with finitely many finite and infinite separatrices, it is sufficient to describe their separatrix configuration. Thus, we have the following cases to analyze:

\medskip

\noindent {\bf Case I.}  If $n=1$, $k\geq1$ is odd and $0<c<1$, then a point $P=(x_0,y_0)$ on the stable separatrices of the equilibrium point $E_1$ has the infinite equilibrium point $I_1$ (see Figure \ref{fign1k>2} $(c)$) as its $\alpha$-limit set. Moreover, a point with $P$ with $y_0>0$ on the unstable separatrix of the equilibrium point $E_1$ has the infinite equilibrium point $I_3^0$ (see Figure \ref{fign1k>2} $(c)$) as its $\omega$-limit set. Conversely, a point $P$ with $y_0<0$ on the unstable separatrix of the equilibrium point $E_1$ has the finite equilibrium point $E_0$ as its $\omega$-limit set. Finally, a point on the unstable separatrix of the infinite equilibrium point $I_2$ has the infinite equilibrium point $I_3^0$ (see Figure \ref{fign1k>2} $(c)$) as its $\omega$-limit set. Therefore the global phase portrait of this system is the one on Figure \ref{gfp} $(i)$. This case has been analysed without taking into account the possible existence of limit cycles.

\medskip

\noindent {\bf Case II.} If $n=1$, $k\geq1$ is odd and $c=1$, the unstable separatrices of the equilibrium point $E_1$ have the same connection configurations as in the previous case. On the other hand, a point $P=(x_0,y_0)$ with $x_0>1$ on the stable separatrix of the equilibrium point $E_1$ has the infinite equilibrium point $I_1$ as its $\alpha$-limit set. Moreover, a point $P$ with $x_0<1$ on the stable separatrix of the equilibrium point $E_1$ has the infinite equilibrium point $I_2$ as its $\alpha$-limit set. Therefore the global phase portrait of this system is the one on Figure \ref{gfp} $(ii)$.

\medskip

\noindent {\bf Case III.} If $n=1$, $k\geq1$ is odd and $1<c<2$, the unstable separatrices of the equilibrium point $E_1$ have the same connection configurations as in the previous cases. A point $P=(x_0,y_0)$ with $x_0>1$ on the stable separatrix of the equilibrium point $E_1$ has the infinite equilibrium point $I_1$ as its $\alpha$-limit set. Now, a point $P$ with $x_0<1$ on the stable separatrix of the equilibrium point $E_1$ has the infinite equilibrium point $I_3^0$ as its $\alpha$-limit set. Moreover, a point on the unstable separatrix of the infinite equilibrium point $I_2$  has the finite equilibrium point $E_0$ as its $\omega$-limit set. The global phase portraits of these systems are the one on Figure \ref{gfp} $(iii-1)$. Finally, if $c\geq2$, the equilibrium point $E_0$ is an stable node, as the local phase portrait of the stable node is topologically equivalent to a stable focus, we also have the phase portrait of Figure \ref{gfp} $(iii-2)$.

\section{Proof of Theorem \ref{thm:traveling}}\label{sec:4}

Consider the partial differential equation \eqref{eq1} under the assumptions: $m = 1$, $n\in\{1,2\}$, $k \in \mathbb N$ and $z \in [0, 1]$. The proof of Theorem \ref{thm:traveling} will be complete if it is proven that there exist solutions $(x(s), y(s))$ of systems \eqref{maineq} whose $\alpha$-limit set is the equilibrium point $E_1 = (1,0)$ and whose $\omega$-limit set is the equilibrium point $E_0 = (0,0)$, still satisfying the restriction $0< x(s) < 1$, for all $s \in \mathbb R$.

Consider any given $c \geq 2$, $n\in\{1,2\}$, and $k \in \mathbb N$. As $c \geq 2$, the equilibrium point $E_0$ is a hyperbolic stable node. The eigenvalues of the Jacobian matrices of the vector fields that define systems \eqref{maineq} at $E_0$ are
\[
\lambda_1 = \dfrac{-c - \sqrt{c^2 - 4}}{2} < 0, \quad \lambda_2 = \dfrac{-c + \sqrt{c^2 - 4}}{2} < 0,
\]
with the respective eigenvectors
\[
V_1 = \left( \dfrac{-c + \sqrt{c^2 - 4}}{2}, 1 \right), \quad V_2 = \left( \dfrac{- c - \sqrt{c^2 - 4}}{2}, 1 \right).
\]
It follows that the slopes of the eigenspaces generated by $V_1$ and $V_2$ at $E_0$ are negative. Thus, the solutions that tend to $E_0$ do so with negative slopes. Remember that, according to Lemma \ref{lemma:0}, systems \eqref{maineq} do not have limit cycles.

On the other hand, as $c \geq 2$, the equilibrium point $E_1$ is a hyperbolic saddle. The eigenvalues of the Jacobian matrices of the vector fields that define systems \eqref{maineq} at $E_1$ are
\[
\mu_1 = \dfrac{1 - c - \sqrt{(c - 1)^2 + 4 n}}{2} < 0, \quad \mu_2 = \dfrac{1 - c + \sqrt{(c - 1)^2 + 4 n}}{2} > 0,
\]
with the respective eigenvectors
\[
W_1 = \left( \dfrac{c - 1 - \sqrt{(c - 1)^2 + 4 n}}{2 n}, 1 \right), \quad W_2 = \left( \dfrac{c - 1 + \sqrt{(c - 1)^2 + 4 n}}{2 n}, 1 \right).
\]
It follows that the slope of the eigenspace generated by $W_1$ at $E_1$ is negative while the slope of the eigenspace generated by $W_2$ at $E_1$ is positive. Therefore, the unstable curve of the saddle point $E_1$ has positive slope at $E_1$.

Consider an initial condition $(x_0, y_0)$, $0 < x_0 < 1$, on the unstable curve of the saddle point $E_1$ and denote by $(x(s), y(s))$ the solution of \eqref{maineq} by this initial condition. From the previous analyses, it follows that the $\alpha$-limit set of $(x(s), y(s))$ is the equilibrium point $E_1 = (1,0)$ and the $\omega$-limit set of $(x(s), y(s))$ is the equilibrium point $E_0 = (0,0)$. Furthermore, this solution $(x(s), y(s))$ satisfies the restriction $0< x(s) < 1$, for all $s \in \mathbb R$. See the red orbits in Figure \ref{gfp} for an illustration. In short, Theorem \ref{thm:traveling} is proved.

As a final observation, it is worth noting that there are other orbits $(x(s), y(s))$ of systems \eqref{maineq} that have finite or infinite equilibrium points as their $\alpha$ and $\omega$-limit sets. This is the case, for example, of the two components of the stable curves of the equilibrium point $E_1$ in all phase portraits of Figure \ref{gfp}. However, due to the restrictions in equation \eqref{eq1}, $0 < x(s) < 1$, these orbits do not give rise to traveling wave solutions of \eqref{eq1}.

\section*{Acknowledgments}

The first author is partially supported by Funda\c c\~ao de Amparo \`a Pesquisa do Estado de S\~ao Paulo (grant number 2019/07316–0), by Funda\c c\~ao de Amparo \`a Pesquisa do Estado de Minas Gerais (grant number APQ-02153-23) and by Conselho Nacional de Desenvolvimento Cient\'ifico e Tecnol\'ogico (grant number 311921/2020-5). The second author is grateful for the warm hospitality of the Universidade Federal de Itajub\'a during the development of this work.

\bigskip

\noindent {\bf Conflicts of Interest}: The authors declare that they have no conflict of interest.

\bigskip

\noindent {\bf Availability of data and material}: Not applicable.

\bigskip

\noindent {\bf Code availability}: Not applicable.

\bigskip

\noindent {\bf Authors' contributions}: Both authors contributed to the study, read and approved the final manuscript.

\end{document}